\documentclass[12pt]{amsart}
\usepackage{latexsym}
\usepackage{amssymb}
\usepackage{amsfonts}
\usepackage{hyperref}

\newtheorem{theorem}{Theorem}
\newtheorem{lemma}[theorem]{Lemma}

\newtheorem{corollary}[theorem]{Corollary}

\theoremstyle{definition}

\theoremstyle{remark}

\numberwithin{equation}{section}

\newcommand{\GL}{{\mathrm {GL}}}

\newcommand{\Aut}{{\mathrm {Aut}}}

\newcommand{\Irr}{{\mathrm {Irr}}}

\newcommand{\Center}{\mathbf{Z}}
\newcommand{\Centralizer}{\mathbf{C}}

\newcommand{\acs}{\mathrm{acs}}
\newcommand{\bC}{{\mathbf{C}}}

\begin{document}

\title[the commuting probability of finite groups]
{On the commuting probability and supersolvability of finite groups}

\author{Paul Lescot (corresponding author) \\
LMRS,CNRS UMR 6085 \\
UFR des Sciences et Techniques, Universit\'{e} de Rouen \\
Avenue de l'Universit\'{e} BP12 \\
76801 Saint-Etienne du Rouvray, France \\
Phone 00 33 (0)2 32 95 52 24 \\
Fax 00 33 (0)2 32 95 52 86 \\
E-mail paul.lescot@univ-rouen.fr \\
\\
\\
Hung Ngoc Nguyen \\
Department of Mathematics \\
The University of Akron \\ 
Akron, Ohio 44325, USA \\
E-mail hungnguyen@uakron.edu \\
\\
\\
Yong Yang \\
Department of Mathematics \\ 
University of Wisconsin--Parkside \\
Kenosha, WI 53141, USA \\
E-mail yangy@uwp.edu}

\subjclass[2010]{Primary 20E45; Secondary 20D10}

\keywords{finite group, conjugacy class, commuting probability}

\date{\today}

\maketitle

\newpage

\begin{abstract} For a finite group $G$, let $d(G)$ denote the probability that a randomly
chosen pair of elements of $G$ commute. We prove that if $d(G)>1/s$
for some integer $s>1$ and $G$ splits over an abelian normal
nontrivial subgroup $N$, then $G$ has a nontrivial conjugacy class
inside $N$ of size at most $s-1$. We also extend two results of
Barry, MacHale, and N\'{\i} Sh\'{e} on the commuting probability in
connection with supersolvability of finite groups. In particular, we
prove that if $d(G)>5/16$ then either $G$ is supersolvable, or $G$
isoclinic to $A_4$, or $G/\Center(G)$ is isoclinic to $A_4$.
\end{abstract}

%%%%%%%%%%%%%%%%%%%%%%%%%%%%%%%%%%%%%%%%%%%%%%%%%%%%%%%%%%%%%%%%%%%%%%

\section{Introduction}

For a group $G$, let $d(G)$ denote the probability that a randomly
chosen pair of elements of $G$ commute. That is,
\[d(G):=\frac{1}{|G|^2}|\{(x,y)\in G\times G\mid xy=yx\}|.\]
This quantity is often referred to as the \emph{commuting
probability} of $G$. The study of the commuting probability of finite
groups dates back to work of Gustafson in the seventies.
In~\cite{Gustafson}, he showed that
\[d(G)=\frac{k(G)}{|G|},\]
where $k(G)$ is the number of conjugacy classes of $G$.

It is clear that $d(G)=1$ if and only if $G$ is abelian. Therefore,
when $d(G)$ is close to 1, one might expect that $G$ is close to
abelian. For instance, it was proved by Gustafson in the same paper
that if $d(G)>5/8$ then $G$ must be abelian. In~\cite{Lescot3}, the
first author classified all groups with commuting probability at
least $1/2$ -- if $d(G)\geq1/2$ then $G$ is isoclinic to the trivial
group, an extraspecial $2$-group, or $S_3$. As a consequence, if
$d(G)>1/2$ then $G$ must be nilpotent. Going further, the first
author proved in~\cite{Lescot1, Lescot2} that $G$ is solvable
whenever $d(G)>1/12$. This was improved by Guralnick and Robinson
in~\cite[Theorem~11]{Guralnick-Robinson} where they showed that if
$d(G)>3/40$ then either $G$ is solvable or $G\cong A_5\times A$ for
some abelian group $A$. In~\cite{Barry-MacHale}, Barry, MacHale, and
N\'{\i} Sh\'{e} proved that $G$ must be supersolvable whenever
$d(G)>1/3$ and pointed out that, since $d(A_4)=1/3$, the bound
cannot be improved.

Two finite groups are said to be \emph{isoclinic} if there exists
isomorphisms between their inner automorphism groups and commutator
subgroups such that these isomorphisms are compatible with the
commutator map, see \S\ref{section3} for the detailed definition.
This concept is weaker than isomorphism and was introduced by
Hall~\cite{Hall} in connection with the enumeration of $p$-groups.
It was shown by the first author~\cite{Lescot3} that the commuting
probability is invariant under isoclinism. It follows that any
group isoclinic to $A_4$ has commuting probability exactly equal to
$1/3$. Our first result highlights the special role of $A_4$ among
non-supersolvable groups with commuting probability greater than
$5/16$. Here and what follows, the center of $G$, as usual, is
denoted by $\Center(G)$.

\begin{theorem}\label{main theorem 3} Let $G$ be a finite group. If $d(G)>5/16$, then
\begin{enumerate}
\item[(i)] $G$ is supersolvable, or
\item[(ii)] $G$ is isoclinic to $A_4$, or
\item[(iii)] $G/\Center(G)$ is isoclinic to $A_4$.
\end{enumerate}
\end{theorem}

\noindent Theorem~\ref{main theorem 3} has two consequences. The
obvious one is the aforementioned result of Barry, MacHale, and
N\'{\i} Sh\'{e}. We would like to note that their proof is somewhat
more complicated and requires a large amount of computations with
GAP~\cite{GAP4}. The second one is less obvious and shows that the
groups isoclinic to $A_4$ are the only non-supersolvable groups of
commuting probability at least $1/3$.

\begin{corollary}\label{corollary} Let $G$ be a finite group with $d(G)\geq 1/3$, then
$G$ is either supersolvable or isoclinic to $A_4$.
\end{corollary}

We remark that, the average size of a conjugacy class of $G$,
denoted by $\acs(G)$, is exactly the reciprocal of $d(G)$.
Therefore, Theorem~\ref{main theorem 3} is equivalent to: that if
$\acs(G)< 16/5$ then either $G$ is supersolvable or $G$ is isoclinic
to $A_4$ or $G/\Center(G)$ is isoclinic to $A_4$. Recently, Isaacs,
Loukaki, and Moret\'{o}~\cite{Isaacs-Loukaki-Moreto} have obtained
some dual results on solvability and nilpotency in connection with
average character degree of finite groups. For instance, they showed
that a finite group is supersolvable whenever its average character
degree is less than $3/2$. It would be interesting if there is a
dual result of Theorem~\ref{main theorem 3} for the average
character degree.

For groups of odd order, it is possible to obtain better bounds. It
was proved in~\cite{Barry-MacHale} that if $G$ is a group of odd
order with $d(G)>11/75$, then $G$ must be supersolvable. Let $C_n$
denote the cyclic group of order $n$. We notice that $(C_5\times
C_5)\rtimes C_3$ is the smallest non-supersolvable group of odd
order. Here we can show that the groups isoclinic to $(C_5\times
C_5)\rtimes C_3$ have commuting probability `substantially' larger
than that of other non-supersolvable groups of odd order.

\begin{theorem}\label{main theorem 2} Let $G$ be a finite group of odd order. If $d(G)>35/243<11/75$, then $G$
is either supersolvable or isoclinic to $(C_5\times C_5)\rtimes
C_3$.
\end{theorem}

Our last result provides a characteristic of certain groups with
`large' commuting probability and therefore can be applied to obtain
the inside structure of these groups. For an example,
see~\S\ref{section4}.

\begin{theorem}\label{main theorem 1} Let $s\geq2$ be an integer and $G$ a finite group with $d(G)>1/s$. Let $N$ be an abelian normal nontrivial subgroup of
$G$ and suppose that $G$ splits over $N$. Then there exists a
nontrivial conjugacy class of $G$ inside $N$ of size at most $s-1$.
In particular, we have either $\Center(G)\neq 1$ or $G$ has a proper
subgroup of index at most $s-1$.
\end{theorem}

Theorems~\ref{main theorem 3}, \ref{main theorem 2} and~\ref{main
theorem 1} are respectively proved in Sections~\ref{section2},
\ref{section3}, and~\ref{section4}. Corollary~\ref{corollary} is
proved at the end of Section~\ref{section2}.

%%%%%%%%%%%%%%%%%%%%%%%%%%%%%%%%%%%%%%%%%%%%%%%%%%%%%%%%%%%%%%%%%%%%%%%%%%%%%%%%%%%%%%%

\section{Groups with commuting probability greater than
5/16}\label{section2}

We will prove Theorem~\ref{main theorem 3} and
Corollary~\ref{corollary} in this section. We first recall a
well-known result of Gallagher~\cite{Gallagher}.

\begin{lemma}\label{Gallagher lemma} If $N$ is a normal subgroup of $G$,
then
\[k(G)\leq k(G/N)k(N)\] and the equality is equivalent to
\[C_{G/N}(gN)=C_G(g)N/N \text{ for each } g\in G.\]
\end{lemma}

This gives an immediate consequence.

\begin{lemma}\label{Gallagher lemma consequence} Let $N$ be a normal subgroup of $G$. Then

\begin{enumerate}
\item[(i)] $d(G)\leq d(G/N)d(N)$,

\item[(ii)] $d(G)=d(G/N)$ if and only if $N$ is abelian and $C_{G/N}(gN)=C_G(g)N/N$ for each $g\in
G$, and

\item[(iii)] if $N\subseteq \Center(G)$ and $d(G)=d(G/N)$, then
$\Center(G/N)=\Center(G)/N$.
\end{enumerate}
\end{lemma}

\begin{proof} (i) and (ii) are consequences of Lemma~\ref{Gallagher
lemma}. We now prove (iii). Assume that $N\subseteq \Center(G)$ and
$d(G)=d(G/N)$. We have $C_{G/N}(gN)=C_G(g)N/N$ for every $g\in G$
and therefore
\begin{align*}gN\in \Center(G/N)&\Leftrightarrow \Centralizer_{G/N}(gN)=G/N\\
&\Leftrightarrow \Centralizer_G(g)N=G\\
&\Leftrightarrow \Centralizer_G(g)=G\\
&\Leftrightarrow g\in \Center(G)
\end{align*}
Therefore, $\Center(G/N)=\Center(G)/N$, as desired.
\end{proof}

Two groups $G$ and $H$ are said to be \emph{isoclinic} if there are
isomorphisms $\varphi:G/\Center(G)\rightarrow H/\Center(H)$ and
$\phi: G'\rightarrow H'$ such that \begin{align*} \text{ if }
\varphi(g_1\Center(G))&=h_1\Center(H)\\ \text{ and }
\varphi(g_2\Center(G))&=h_2\Center(H),\\ \text{ then }
\phi([g_1,g_2])&=[h_1,h_2].\end{align*} This concept is weaker than
isomorphism and was introduced by Hall in~\cite{Hall} as a
structurally motivated classification for finite groups,
particularly for $p$-groups. It is well-known that several
characteristics of finite groups are invariant under isoclinism and
in particular supersolvability is one of those,
see~\cite{Bioch-Waall}. Furthermore, it is proved in~\cite{Lescot3}
that the commuting probability is also invariant under isoclinism.

A stem group is defined as a group whose center is contained inside
its derived subgroup. It is known that every group is isoclinic to a
stem group and if we restrict to finite groups, a stem group has the
minimum order among all groups isoclinic to it, see~\cite{Hall} for
more details. The following lemma plays an important role in the
proof of Theorems~\ref{main theorem 3} and~\ref{main theorem 2}.

\begin{lemma}\label{isoclinic lemma} For every finite group $G$,
there is a finite group $H$ isoclinic to $G$ such that $|H|\leq |G|$
and $\Center(H)\subseteq H'$.
\end{lemma}

The next lemma will narrow down the possibilities for the commutator
subgroup of a finite group with commuting probability greater than
$5/16$.

\begin{lemma}\label{|G'|<9} Let $G$ be an finite group with $d(G)>5/16$. Then
$|G'|< 12$.
\end{lemma}

\begin{proof} Let $\Irr_2(G)$ denote the set of nonlinear irreducible complex
characters of $G$. Then, as $G$ has exactly $[G:G']$ linear
characters, we have $|\Irr_2(G)|=k(G)-[G:G']$ where $k(G)$ is the
number of conjugacy classes of $G$. We obtain
$$|G|=[G:G']+\sum_{\chi\in\Irr_2(G)} \chi(1)^2\geq
[G:G']+4(k(G)-[G:G']).$$ As $d(G)=k(G)/|G|$, it follows that
$$\frac{1}{|G'|}+4(d(G)-\frac{1}{|G'|})\leq 1.$$ Using the
hypothesis $d(G)>5/16$, we deduce that $|G'|<12$.
\end{proof}

We are now ready to prove Theorem~\ref{main theorem 3}.

\begin{proof}[Proof of Theorem~\ref{main theorem 3}] Assume that $G$ is a finite group
with $d(G)>5/16$ and $G$ is not supersolvable. We aim to show that
either $G$ is not isoclinic to $A_4$ or $G/\Center(G)$ is isoclinic
to $A_4$. Since commuting probability and supersolvability are both
invariant under isoclinism, using Lemma~\ref{isoclinic lemma}, we
can assume that $\Center(G)\subseteq G'$. Indeed, if
$\Center(G)=G'$, then $G$ is nilpotent which violates our
assumption. So we assume furthermore that $\Center(G)\varsubsetneq
G'$. Recall that $d(G)>5/16$ and hence $|G'|\leq11$ by
Lemma~\ref{|G'|<9}. We note that $G'$ is noncyclic as $G$ is not
supersolvable.

First we remark that $S_3,D_8$ as well as $D_{10}$ have a cyclic,
characteristic, non-central subgroup and hence it is well-known that
they cannot arise as commutator subgroups, see~\cite{MacHale-Murchu}
for instance. Thus we are left with the following possibilities of
$G'$.

\medskip

\textbf{Case $G'\cong C_2\times C_2$}: If $\Center(G)\cong C_2$ then
the normal series $1<\Center(G)<G'<G$ implies that $G$ is
supersolvable. So we assume that $\Center(G)=1$. Thus $G'$ is a
minimal normal subgroup of $G$. Now, since $G$ is not supersolvable,
$G'\nsubseteq \Phi(G)$, see~\cite{Huppert} for instance. Therefore,
$G'$ is not contained in a maximal subgroup $G$, say $M$. We have
$G=G'M$. Also, as $G'$ is abelian, we see that $G'\cap M \lhd G$.
Now the minimality of $G'$ and the fact that $G'$ is not contained
in $M$ imply that $G'\cap M=1$. This means $G$ splits over $G'$ or
equivalently $G\cong G'\rtimes M$. Thus, as $M\cong G/G'$ is abelian
and $\Center(G)=1$, we deduce that $\Centralizer_M(G')=1.$ It
follows that
$$M\leq \Aut(G')\cong\Aut(C_2\times C_2)\cong S_3,$$ and hence
$$M\cong C_2 \text{ or } M\cong C_3.$$ In the former case, $|G|=8$
and $G$ would be nilpotent, a contradiction. In the latter case,
$G\cong A_4$ and we are done.

\medskip

\textbf{Case $G'\cong C_3\times C_3$}: As in the previous case, we
can assume $\Center(G)=1$, $G'$ is a minimal normal subgroup of $G$,
and $G\cong G'\rtimes M$. Here $M$ is an abelian subgroup of
$\Aut(G')\cong \GL_2(3)$. Consulting the list of subgroups of
$\GL_2(3)$ reveals: $M\cong C_2,C_3,C_4, C_2\times C_2, C_6$, or
$C_8$. However, it is routine to check that all these possibilities
result in either $G$ is supersolvable or $d(G)\leq5/16$.

\medskip

In the remaining cases, we let $N$ be a minimal normal subgroup of
$G$ with $N\subseteq G'$. Recall that $\Center(G)\varsubsetneq G'$
and so in the case $\Center(G)\neq 1$, we can even take $N\subseteq
\Center(G)$.

\medskip

\textbf{Case $G'\cong Q_8$, $C_4\times C_2$, or $C_2\times C_2\times
C_2$} and $N\cong C_2$: If $N<\Center(G)$ then the normal series
$1<N<\Center(G)<G'<G$ would imply that that $G$ is supersolvable, a
contradiction. Thus $\Center(G)=N\cong C_2$. Also, as $G$ is not
supersolvable, $G/\Center(G)$ is not supersolvable as well. It
follows that $(G/\Center(G))'$ is not cyclic so that
\[(\frac{G}{\Center(G)})'\cong C_2\times C_2.\]
By Lemma~\ref{Gallagher lemma consequence}, we know that $d(G/N)\geq
d(G)>5/16$. Now we are in the first case with $G/\Center(G)$
replacing $G$. Therefore, we conclude that either $G/\Center(G)$ is
supersolvable, a contradiction, or $G/\Center(G)$ is isoclinic to
$A_4$, as desired.

\medskip

\textbf{Case $G'\cong C_4\times C_2$, or $C_2\times C_2\times C_2$}
and $N\cong C_2\times C_2$: Then $G'/N\cong C_2$ is a normal
subgroup of $G/N$. In particular, $G'/N \subseteq \Center(G/N)$ and
as $G/G'$ is abelian, we deduce that $G/N$ is supersolvable. As in
the previous case, by using the non-supersolvability of $G$, we
deduce that $G\cong N\rtimes M$ where $M$ is a maximal subgroup of
$G$. It then follows that $C_2\cong G'\cap M\vartriangleleft M$ and
hence $G'\cap M$ centralizes $M$, whence $G'\cap M$ centralizes $G$.
This implies that $\Center(G)\neq 1$, which in turn implies that
$N=\Center(G)$ since $N\subseteq \Center(G)\subsetneq G'$. This
violates the minimality of $N$.

\medskip

\textbf{Case $N=G'\cong C_2\times C_2\times C_2$}: Since
$\Center(G)\varsubsetneq G'$, we obtain $\Center(G)=1$. As before,
we can show that $G\cong G'\rtimes M$ where $M$ is an abelian
subgroup of $\Aut(G')\cong \GL_3(2)$. This implies that $M\cong C_2,
C_3, C_4, C_2\times C_2$, or $C_7$. The cases $M\cong C_2, C_4$ or
$C_2\times C_2$ would imply that $G$ is a $2$-group. The case
$M\cong C_7$ would imply that $d(G)=1/7$. Finally, the case $M\cong
C_3$ implies that $G\cong C_2\times A_4$ and hence $G'\cong
C_2\times C_2$, which is the final contradiction.
\end{proof}

Now we prove Corollary~\ref{corollary}.

\begin{proof}[Proof of Corollary~\ref{corollary}] Assume, to the contrary, that the statement is false and let $G$ be a minimal counterexample. Again we know
that $\Center(G)\varsubsetneq G'$ and as in the proof of
Lemma~\ref{|G'|<9}, we also have $|G'|\leq9$. Using
Theorem~\ref{main theorem 3}, we deduce that $G/\Center(G)$ is
isoclinic to $A_4$ and $\Center(G)$ is nontrivial. In particular,
$G'/\Center(G)\cong A_4'\cong C_2\times C_2$ and hence
$\Center(G)\cong C_2$. Using Lemma~\ref{Gallagher lemma
consequence}(i), we have
\[\frac{1}{3}=d(\frac{G}{\Center(G)})\geq d(G)\geq \frac{1}{3}.\]
Therefore $d(G)=d(G/\Center(G))$, which implies that
$\Center(G/\Center(G))=\Center(G)/\Center(G)=1$ by
Lemma~\ref{Gallagher lemma consequence}(iii). It follows that
$G/\Center(G)\cong A_4$ and as $|\Center(G)|=2$, we have
$G=C_2\times A_4$. This violates the assumption that
$\Center(G)\varsubsetneq G'$.
\end{proof}

%%%%%%%%%%%%%%%%%%%%%%%%%%%%%%%%%%%%%%%%%%%%%%%%%%%%%%%%%%%%%%%%%%%%%%%%%%%%%%%%%%%%%%%%%%%%%%%%%%%%%%%%%%%%%%%%%%%%%%%%%%%%%%%%%%%%%%%

\section{Groups of odd order}\label{section3}

We will prove Theorem~\ref{main theorem 2} in this section. As in
Section~\ref{section2}, we narrow down the possibility for the order
of the commutator subgroup of a group in consideration.

\begin{lemma}\label{|G'|<27} Let $G$ be an odd order finite group with $d(G)>35/243$. Then
$|G'|<27$.
\end{lemma}

\begin{proof} We repeat some of the arguments in the proof of Lemma~\ref{|G'|<9}.
Recall that $\Irr_2(G)$ denotes the set of nonlinear irreducible
complex characters of $G$ and we have $|\Irr_2(G)|=k(G)-[G:G']$.
Since $|G|$ is odd, every character in $\Irr_2(G)$ has degree at
least 3. We obtain
$$|G|=[G:G']+\sum_{\chi\in\Irr_2(G)} \chi(1)^2\geq
[G:G']+9(k(G)-[G:G']),$$ and therefore
$$\frac{1}{|G'|}+9(d(G)-\frac{1}{|G'|})\leq 1.$$ Since $d(G)>35/243$,
it follows that $|G'|<27$, as wanted.
\end{proof}

\begin{proof}[Proof of Theorem~\ref{main theorem 2}] We argue by
contradiction and let $G$ be a minimal counterexample. Since
commuting probability and supersolvability are both invariant under
isoclinism, using Lemma~\ref{isoclinic lemma}, we can assume that
$\Center(G)\subseteq G'$. Now $G$ is a non-supersolvable group of
odd order with $d(G)>35/243$. Applying Lemma~\ref{|G'|<27}, we have
$|G'|<27$ so that $G'$ is a noncyclic odd order group of order at
most $25$.

We choose a minimal normal subgroup $N$ of $G$ with $N\subseteq G'$.
and note that $N$ is elementary abelian. By Lemma~\ref{Gallagher
lemma consequence}, we have $d(G/N)\geq d(G)>35/243$ so that $G/N$
is supersolvable or isoclinic to $(C_5\times C_5)\rtimes C_3$ by the
minimality of $G$.

First we show that the case where $G/N$ is isoclinic to $(C_5\times
C_5)\rtimes C_3$ cannot happen. Assume so. Then $$G'/N=(G/N)'\cong
((C_5\times C_5)\rtimes C_3)'=C_5\times C_5,$$ which implies that
$|G'|$ is at least $50$, a contradiction. We conclude that $G/N$ is
supersolvable. In particular, if $N$ is cyclic then $G$ is
supersolvable and we have a contradiction. Note that, when $G'\ncong
C_3\times C_3$ and $C_5\times C_5$, a routine check on groups of odd
order at most $25$ shows that $N$ must be cyclic. Thus, it remains
to consider the cases where $G'\cong C_3\times C_3$ or $G'\cong
C_5\times C_5$ and $G'$ is a minimal normal subgroup of $G$.

If $G'\subseteq \Phi(G)$ then $G$ is nilpotent and we are done.
Therefore we can assume that $G'\nsubseteq \Phi(G)$ and, as in the
proof of Theorem~\ref{main theorem 3}, we see that $G$ splits over
$G'$ by a maximal subgroup $M$ of $G$: $$G\cong G'\rtimes M.$$

\noindent Since $\Center(G)\subseteq G'$, we must have
$\Center(G)=1$ and, as $M\cong G/G'$ is abelian, we deduce that
$\Centralizer_M(G')=1.$ It follows that $$M\leq \Aut(G').$$ First we
assume that $G'\cong C_3\times C_3$, then $M$ is an abelian subgroup
of odd order of $\Aut(C_3\times C_3)=\GL_2(3)$. This forces $M\cong
C_3$, which implies that $|G|=|G'||M|=9\cdot 3=27$ and hence $G$ is
nilpotent, a contradiction. Next we assume that $G'\cong C_5\times
C_5$. Arguing similarly, we see that $M$ is an abelian subgroup of
odd order of $\GL_2(5)$. This forces $M\cong C_3, C_5$ or $C_{15}$.
The case $M\cong C_5$ would imply that $G$ is nilpotent whereas the
case $M\cong C_{15}$ would imply that $d(G)=23/375$, which is a
contradiction. We conclude that $M\cong C_3$ so that $G\cong
(C_5\times C_5)\rtimes C_3$.
\end{proof}

%%%%%%%%%%%%%%%%%%%%%%%%%%%%%%%%%%%%%%%%%%%%%%%%%%%%%%%%%%%%%%%%%%%%%%%%%%%%%%%%%%%%%%%%%%%%%%%%%%%%%%%%%%%%%%%%%%%%%%%%%%%%%%%%%%%%%%%%

\section{A conjugacy class size theorem}\label{section4}

In this section, we prove Theorem~\ref{main theorem 1} and then give
an example showing how one can obtain some properties of certain
groups with `large' commuting probability.

\begin{proof}[Proof of Theorem~\ref{main theorem 1}] Assume, to the contrary, that all the nontrivial orbits of the conjugacy action of $G$ on
$N$ have size at least $s$. Since $G$ splits over $N$, let $G=HN$
where $H \cap N=1$. We denote $C=\bC_H(N)$ and clearly $C \lhd H$.
Every element of $G$ can be written uniquely as $ha$ where $h \in H$
and $a \in N$. We now examine the class sizes in $G$.

First, let $g=ha$ with $a \neq 1$. Sine $a$ is in an orbit of $H$ on
$N$ of size greater than or equal to $3$, we can find $s-1$ other
elements $a_2, a_3,..., a_s$ in the orbit of $a$. Therefore there
exist $t_2, t_3,...,t_s \in H$ such that $a^{t_i}=a_i$ for $2\leq
i\leq s$. Thus
$$g=ha, g^{t_2}=h^{t_2} a_2, g^{t_3}=h^{t_3} a_3,..., g^{t_s}=h^{t_s}a_s$$ are
$s$ different elements in the conjugacy class of $g$. We now have
\begin{equation}\label{equation 1}
\text{every conjugacy class of an element outside } H \text{ has
size at least } s.
\end{equation}

It remains to consider the conjugacy classes of elements of $H$. Let
$g=h_1$ for some $h_1 \in H$. If $h_1 \not \in C$, then there exits
some $a \in N$ which is not fixed by $h_1$. Thus
$$a h_1 a^{-1}=h_1 h_1^{-1} a h_1 a^{-1}=h_1 a_1, \text{ where }
a_1=h_1^{-1}ah_1a^{-1} \neq 1.$$ By the previous paragraph, we know
that there are $t_2, t_3,...,t_s \in H$ such that $$g=h_1, h_1 a_1,
h_1^{t_2} a_2, h_1^{t_3} a_3,...,h_1^{t_s}a_s$$ are $s+1$ different
elements in the conjugacy class of $g$, where $a_i$'s are nontrivial
and distinct. Thus $g$ is in a conjugacy class of size at least
$s+1$.

If $h_1 \neq h_2$, $h_1, h_2 \in H \backslash C$, and $h_1, h_2$ are
in the same conjugacy class, then we know that $$h_1, h_1 a_1,
h_1^{t_2} a_2, h_1^{t_3} a_3,...,h_1^{t_s}a_s, h_2
$$ are $s+2$ distinct elements in the same conjugacy class. This
implies in general that if $h_1, h_2, \dots, h_t$ are distinct
elements in $H \backslash C$ and $h_1, h_2, \dots, h_t$ are in the
same conjugacy class, then the size of the conjugacy class is
greater than or equals to
$$s+t.$$

Denote $k=|H|/|C|$, then there are $(k-1)|C|$ elements in $H
\backslash C$. We consider all the conjugacy classes of $G$ which
contain some elements in $H$. Suppose all the elements in $H
\backslash C$ belong to $n$ different conjugacy classes and each
conjugacy class contains $t_1, \dots, t_n$ elements in $H \backslash
C$ respectively. Then $$\sum_{i=1}^n t_i=|H\backslash C|=(k-1)|C|$$
and the sum of sizes of these $n$ classes is at least
$$\sum_{i=1}^n (s+t_i)=ns+\sum_{i=1}^n t_i.$$ Therefore, the average size of a conjugacy class of an element in $H$ is at least
\[\frac {ns+\sum_{i=1}^n t_i+|C|}{n+|C|}=\frac
{ns+k|C|}{n+|C|}.\]

Since $C$ acts trivially on $N$ and $|H/C|=k$, the conjugacy action
of $G$ on $N$ has orbit of size at most $k$. Since $N$ is nontrivial
and every nontrivial orbit of $G$ on $N$ has size at least $s$, we
deduce that $k\geq s$. It follows that $$\frac {ns+k|C|}{n+|C|}\geq
\frac {ns+s|C|}{n+|C|}= s$$ and hence
\begin{equation}\label{equation 2}
\text{the average size of a class of an element in } H \text{ is at
least } s.
\end{equation}
Combining (\ref{equation 1}) and~(\ref{equation 2}), we conclude
that the average class size of $G$ is at least s, which violates the
hypothesis that $d(G)>1/s$.
\end{proof}

The following is an application of Theorem~\ref{main theorem 1} to
the study of finite groups with commuting probability greater than
$1/3$.

\begin{corollary}\label{lemma d(G)<1/3} Let $G=(C_2\times C_2)\rtimes H$
and assume that $d(G)>1/3$. Then there exists a nontrivial element
of $C_2\times C_2$ that is fixed under the conjugation action of
$H$. In particular, $\Center(G)\neq 1$.
\end{corollary}

\begin{proof} By Theorem~\ref{main theorem 1}, the group $H$ has a
nontrivial orbit of size at most $2$ on $C_2\times C_2$. If this
orbit has size 1 then we are done. Otherwise, it has size 2 and
hence the other orbit must have size 1, as wanted.
\end{proof}

%%%%%%%%%%%%%%%%%%%%%%%%%%%%%%%%%%%%%%%%%%%%%%%%%%%%%%%%%%%%%%%%%%%%%%%%%%%%%%

\end{document}